\documentclass[11pt]{amsart}
\usepackage[margin=1in]{geometry}
\usepackage{amsmath}
\usepackage{amssymb}
\usepackage{amsthm}
\usepackage{graphicx}
\usepackage{xcolor}
\usepackage{caption}
\usepackage{float}
\usepackage[labelformat=simple]{subcaption}
\usepackage{comment}
\usepackage{hyperref}

\newcommand{\om}{\omega}
\newcommand{\bom}{\overline{\omega}}
\newcommand{\som}{\sqrt{\omega}}
\newcommand{\bsom}{\overline{\sqrt{\omega}}}

\newcommand{\mb}{\mathbb}
\newcommand{\R}{\mathbb{Z}[\som^{\pm 1}]}

\newcommand{\krem}{\widetilde{\tau}_D}

\newcommand{\vtex}{v}

\newtheorem{thm}{Theorem}[section]

\newtheorem{cor}[thm]{Corollary}
\newtheorem{defn}[thm]{Definition}

\newcommand{\theoremname}{Theorem:}
\theoremstyle{remark}
\newtheorem{remark}[thm]{Remark}
\newtheorem{ex}[thm]{Example}

\DeclareMathOperator{\lk}{lk}
\DeclareMathOperator{\sign}{sign}
\DeclareMathOperator{\wri}{wr}

\title{A proof of the Kashaev signature conjecture}

\author[J. Liu]{Jessica Liu}
\address{Department of Mathematics\\ University of Toronto\\ Toronto, Ontario, Canada}
\email{jessliu@math.toronto.edu}
\date{September 2023}

\begin{document}

\begin{abstract} In 2018 Kashaev introduced a diagrammatic link invariant conjectured to be twice the Levine-Tristram signature. If true, the conjecture would provide a simple way of computing the Levine-Tristram signature of a link by taking the signature of a real symmetric matrix associated with a corresponding link diagram. This paper establishes the conjecture using the Seifert surface definition of the Levine-Tristram signature on the disjoint union of an oriented link and its reverse. The proof also reveals yet another formula for the Alexander polynomial.
\end{abstract}
\maketitle
\tableofcontents

\section{Introduction}

The Levine-Tristram signature $\sigma_L: S^1 \to \mb{Z}$ of a link $L$ is a classical invariant that has been actively studied since its introduction in the sixties, yet is still not fully understood. While original definitions by Levine \cite{lev} and Tristram \cite{trist} use Seifert matrices, $\sigma_L$ can also be defined using various pairings on the Alexander module, as well as through 4-dimensional topology. It is closely related to the Alexander polynomial and its roots, and provides lower bounds for many topological invariants such as the slice \cite{kauf} and doubly slice \cite{powell} genus. There is also a connection with the Burau representation \cite{gg}, as well as a multivariable generalization \cite{cim}. A survey of some results regarding the Levine-Tristram signature can be found in \cite{con}.

Through a recent attempt to understand the metaplectic invariants of Goldschmidt-Jones \cite{gj}, Kashaev introduced a link invariant  defined using a simple algorithm on link diagrams which he conjectured also computes the Levine-Tristram signature \cite{kash}. If true, Kashaev's conjecture would provide a diagrammatic way of computing the Levine-Tristram signature using the signature of a real matrix. The purpose of this paper is to present a method of obtaining Kashaev's invariant using the original Seifert surface definition of the Levine-Tristram signature, making evident the relationship between the two and thereby proving Kashaev's conjecture. 

The idea behind the proof is as follows. For an oriented link $L \in S^3$, let $-L$ denote the same link with opposite orientation. By picking a suitable Seifert surface $\Sigma$ for $L \sqcup -L$ and basis for its first homology, one sees that Kashaev's invariant on $L$ is exactly the Levine-Tristram signature of $L \sqcup -L$, which is twice that of $L$.

It is not surprising that Kashaev's definition for signature also yields formulas for the Alexander polynomial $\Delta_L(t)$. While $\Delta_{L\sqcup-L}$ is always 0, observe that $\Delta_{L \# -L}=\Delta_L^2$. A slight modification of the construction for $L \sqcup -L$ produces Seifert matrices for the connect sum $L \# -L$, which gives the same results about signature, but also reveals another method of computing $\Delta_L$. 

The relationship with $\Delta_L$ also appears in recent work by Cimasoni and Ferretti \cite{liv}, where they prove the formula for $\Delta_L$ for links, as well as the Kashaev signature conjecture in the case of definite knots (i.e. knots with Seifert matrix $A$ such that $A+A^T$ is positive or negative definite), by relating Kashaev's invariant to the Kauffman model of the Alexander polynomial ~\cite{kaufalex} and using the Gordon-Litherland description of the Levine-Tristram signature at $-1$ ~\cite{gl}. Following remarks 3.1 and 3.2 in ~\cite{liv}, their work implies that the results in this paper also provide proofs of both the Kauffman model of the Alexander polynomial as well as the Gordon-Litherland formulation of the signature at -1. Furthermore, Kashaev's invariant may be thought of as an extension of the latter to the full Levine-Tristram signature.\\

\emph{Acknowledgements} This paper is a result of conversations with Dror Bar-Natan, who noticed the relationship with the Alexander polynomial and suggested looking at links whose Alexander polynomial is the square of the original. Dror also provided countless useful discussions and insights, as well as code for experimentation and verification. I am extremely grateful for all of his support. I am also thankful to David Cimasoni, Livio Ferretti, and Rinat Kashaev for useful discussions. This work is partially supported by NSERC CGS-D and MS-FSS.\\

\section{Review of the Levine-Tristram signature}

A Seifert surface for an oriented link $L$ is a compact, connected, orientable surface embedded in $S^3$ whose boundary is $L$. Let $\Sigma$ be a Seifert surface for $L$, and consider a regular neighbourhood of $\Sigma$ homeomorphic to $\Sigma \times [-1, 1]$, where $\Sigma$ is identified with $\Sigma \times \{0\}$. Let $i^-: H_1(\Sigma; \mb{Z}) \to H_1(S^3 \setminus \Sigma; \,\mb{Z})$ denote a pushoff in the negative direction, which sends the homology class of a curve $\gamma$ to that of $\gamma \times \{-1\}$. The \emph{Seifert form} on $H_1(\Sigma)$ is \begin{align*}H_1(\Sigma) \times H_1(\Sigma) &\to \mathbb{Z} \\ (a, b) &\mapsto \lk(i^{-1}(a), b)\end{align*} where $\lk$ denotes the linking number. A \emph{Seifert matrix} is any matrix which represents the Seifert form. Given a complex unit $\omega \in S^1$ and real matrix $A$, observe that the matrix $$(1-\om)A - (1-\bom)A^T$$ is Hermitian, and thus has a well-defined \emph{signature}: the number of positive eigenvalues minus the number of negative eigenvalues.

\begin{defn} The Levine-Tristram signature of an oriented link $L \in S^3$ is the map $\sigma_L: S^1 \to \mb{Z}$ given by $$\sigma_L(\omega) = \sign((1-\omega)A + (1-\bom)A^T)$$ where $A$ is any Seifert matrix for $L$ and $\sign$ denotes the signature.
\end{defn}

The signature of a Hermitian matrix does not change if the matrix is enlarged by adding a row and column of zeros, thus the Seifert matrix $A$ in the definition of $\sigma_L$ need not be taken with respect to a basis of $H_1(\Sigma)$, but rather any set of generators. On the other hand, the closely related Alexander polynomial $\Delta_L(t) = \det(t A-A^T)$ becomes zero if additional generators are added, so it is important that $A$ is taken with respect to a basis when considering the Alexander polynomial.\\

\section{Review of Kashaev's invariant}
An oriented link diagram $D$ has the structure of a degree-4 planar graph by viewing the crossings of the diagram as the vertices of the graph. Let $F_D$ denote the faces of $D$ when viewed as a planar graph in this way. Associate to $D$ a symmetric $|F_D| \times |F_D|$ matrix $\tau_D$ with values in $\mb{Z}[x]$ as follows: 

For a vertex $\vtex$, let $\tau_\vtex$ be the $|F_D| \times |F_D|$ matrix which is zero everywhere except in the $4 \times 4$ minor corresponding to the faces adjacent to $\vtex$, where its values are given by the matrix in figure ~\ref{fig:kash}.

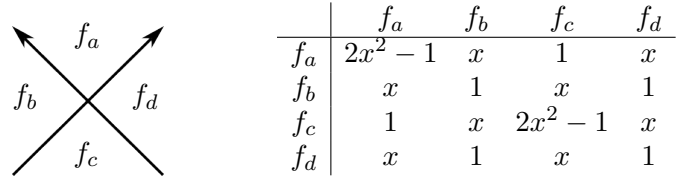
\begin{figure}[H]

\begin{picture}(100,70)
\put(-50,0){\includegraphics[width = 2cm]{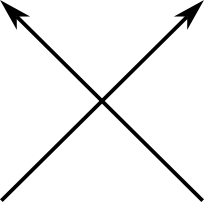}}
\put(-27,5){$f_c$}
\put(-27, 50){$f_a$}
\put(-50, 27){$f_b$}
\put(-5, 27){$f_d$}
\put(50, 30){$\begin{array}{c|cccc} & f_a & f_b & f_c & f_d\\
\hline
f_a&2x^2-1 & x & 1 & x\\
f_b&x&1&x&1\\
f_c&1&x&2x^2-1&x\\
f_d&x&1&x&1\end{array}$}
\end{picture}\vspace*{0.2cm}
\caption{A vertex $v$ and its four adjacent faces, along with the corresponding $4 \times 4$ minor of $\tau_v$}
\label{fig:kash}
\end{figure} 
\begin{defn}
The Kashaev matrix $\tau_D$ of an oriented link diagram $D$ is the following sum over the vertices of $D$: $$\tau_D := \sum_{\vtex} \epsilon(v)\tau_\vtex$$ where $\epsilon(\vtex) = 1$ for a vertex corresponding to a positive crossing and $-1$ for a vertex corresponding to a negative crossing. We call these positive and negative vertices for simplicity.  
\end{defn}

Note that $\tau_D$ is actually a matrix with entries in $\mb{Z}[2x]$ -- the only occurrences of $x$ are in entries corresponding to faces that share an edge, or in entries on the diagonal corresponding to a face with itself. Faces that share an edge always share an even number of vertices, and any face occurs next to a vertex in the position of $f_a$ and $f_c$ in figure ~\ref{fig:kash} an equal number of times. 

Note also that it is possible for two of the surrounding faces of $\vtex$ to be part of the same face, as in the vertex between $f_2, f_5,$ and $f_6$ in example ~\ref{ex:kash}. In this case, $\tau_\vtex$ is nonzero in a $3 \times 3$  minor, and the values of the shared face are given by summing the corresponding values of the two surrounding faces as if they were distinct faces. 

\begin{ex}An example of a diagram $D$ and its Kashaev matrix $\tau_D$: \label{ex:kash}

\begin{picture}(300, 100)
\put(0,0){\includegraphics[width=4cm]{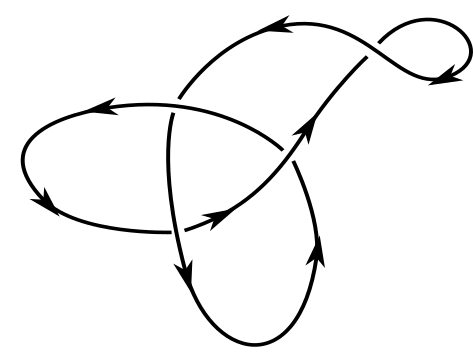}}
\put(20, 45){$f_1$}
\put(57, 20){$f_4$}
\put(47, 45){$f_3$}
\put(60, 65){$f_2$}
\put(98, 72){$f_5$}
\put(90, 30){$f_6$}
\put(145,45){$\begin{array}{c|cccccc}
&f_1&f_2&f_3&f_4&f_5&f_6\\
\hline
f_1&4x^2-2&1&2x&1&0&2x\\
f_2&1&4x^2-3&2x&1&-1&0\\
f_3&2x&2x&3&2x&0&3\\
f_4&1&1&2x&4x^2-2&0&2x\\
f_5&0&-1&0&0&-1&-2x\\
f_6&2x&0&3&2x&-2x&-4x^2+3\end{array}$}
\end{picture}
\end{ex}

While $\tau_D$ is not an invariant of links, its signature is after a correction by the writhe. The following theorem is due to Kashaev. 

\begin{thm}[Kashaev \cite{kash}]\label{thm:ogkash} Let $L$ be an oriented link represented by the diagram $D$, and let $\wri_D$ denote the writhe of $D$. For any real value of $x$, $$\sign(\tau_D) - \wri_D$$ is an invariant of $L$.  
\end{thm} The proof of theorem \ref{thm:ogkash} in \cite{kash} uses a modified notion of $S$-equivalence to show the invariance under Reidemeister moves and requires $x \neq -\frac{1}{2}$. Our interpretation in the next section makes clear that $\sign(\tau_D) - \wri_D$ computes the Levine-Tristram signature when $x \in (-1,1)$, hence it is an invariant for all real $x$.\\ 

\section{From the Levine-Tristram signature to Kashaev's invariant}

This section is dedicated to the construction of Kashaev's invariant using the Levine-Tristram signature. We first state Kashaev's conjecture. 

\begin{thm}[Kashaev's conjecture for signatures]\label{thm:kash}
Given an oriented link $L$ represented by diagram $D$, the Levine-Tristram signature $\sigma_L$ can be computed by $$\sigma_L(\om) =\frac{1}{2}\sign(\tau_D - \wri_D)$$ under the identification  $2x = \som + \bsom = 2\Re(\som)$.
\end{thm} 

The structure the proof is as follows. Construct a Seifert surface for $L \sqcup -L$ whose first homology is generated by classes of curves corresponding to the faces and vertices of $D$. Using a Seifert matrix $A$ with respect to these generators, we see that $Q = (1-\om)A + (1-\bom)A^T$ is congruent to a block diagonal matrix with two blocks where one block corresponds to vertices and has signature $-\wri(D)$ and the other block corresponds to faces and, with a scaling of the generators, is exactly $\tau_D$. Since $\sign(Q) = \sigma_{L \sqcup -L} = 2\sigma_L$, the proof is complete.

\begin{proof} We start with the construction of a Seifert surface for $L \sqcup -L$. From a diagram $D$ of $L$, draw a diagram for $L \sqcup -L$ as follows:
\begin{enumerate}
\item At each crossing of $D$, draw a corresponding crossing for $-L$ a bit ``above and behind'' the existing crossing in $D$, as in figure ~\ref{fig:aboveandbehind}.
\begin{figure}[h!]
\begin{subfigure}{0.3\textwidth}
\centering
\begin{picture}(100,35)
\put(0,0){\includegraphics[width=1cm]{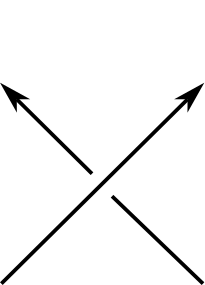}}
\put(35,10){$\longrightarrow$}
\put(60,0){\includegraphics[width=1cm]{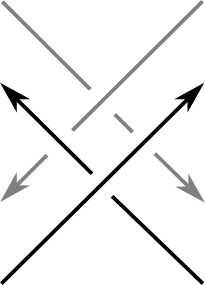}}
\end{picture}
\caption{A positive crossing}
\end{subfigure}
\hspace{2cm}
\begin{subfigure}{0.3\textwidth}
\centering
\begin{picture}(100,30)
\put(0,0){\scalebox{-1}[1]{\includegraphics[width=1cm]{Lcross.png}}}
\put(35,10){$\longrightarrow$}
\put(60,0){\scalebox{-1}[1]{\includegraphics[width=1cm]{negLcross.png}}}
\end{picture}
\caption{A negative crossing}
\end{subfigure}
\caption{A crossing of $D$ (black) with the corresponding crossing (grey) for $-L$}
\label{fig:aboveandbehind}
\end{figure}

\item Connect the new crossings with edges that follow along the corresponding edges in $D$, possibly creating an extra crossing between $L$ and $-L$ along each edge of $D$. See figure ~\ref{fig:LsqcupL} for an example.

\begin{figure}[H]
\centering
\begin{picture}(400,90)
\put(0,0){\includegraphics[width=4cm]{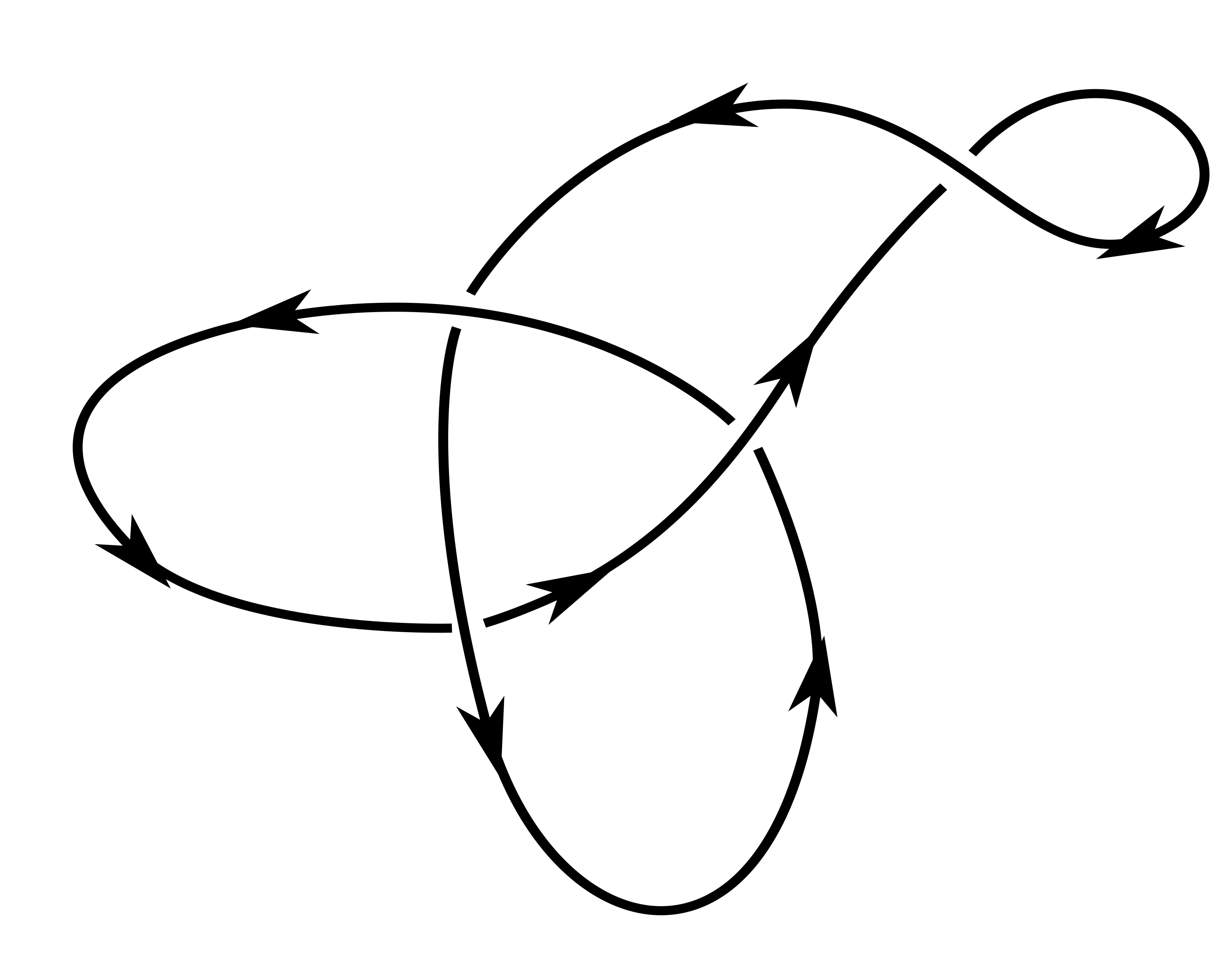}}
\put(100,40){$\longrightarrow$}
\put(130,0){\includegraphics[width=4cm]{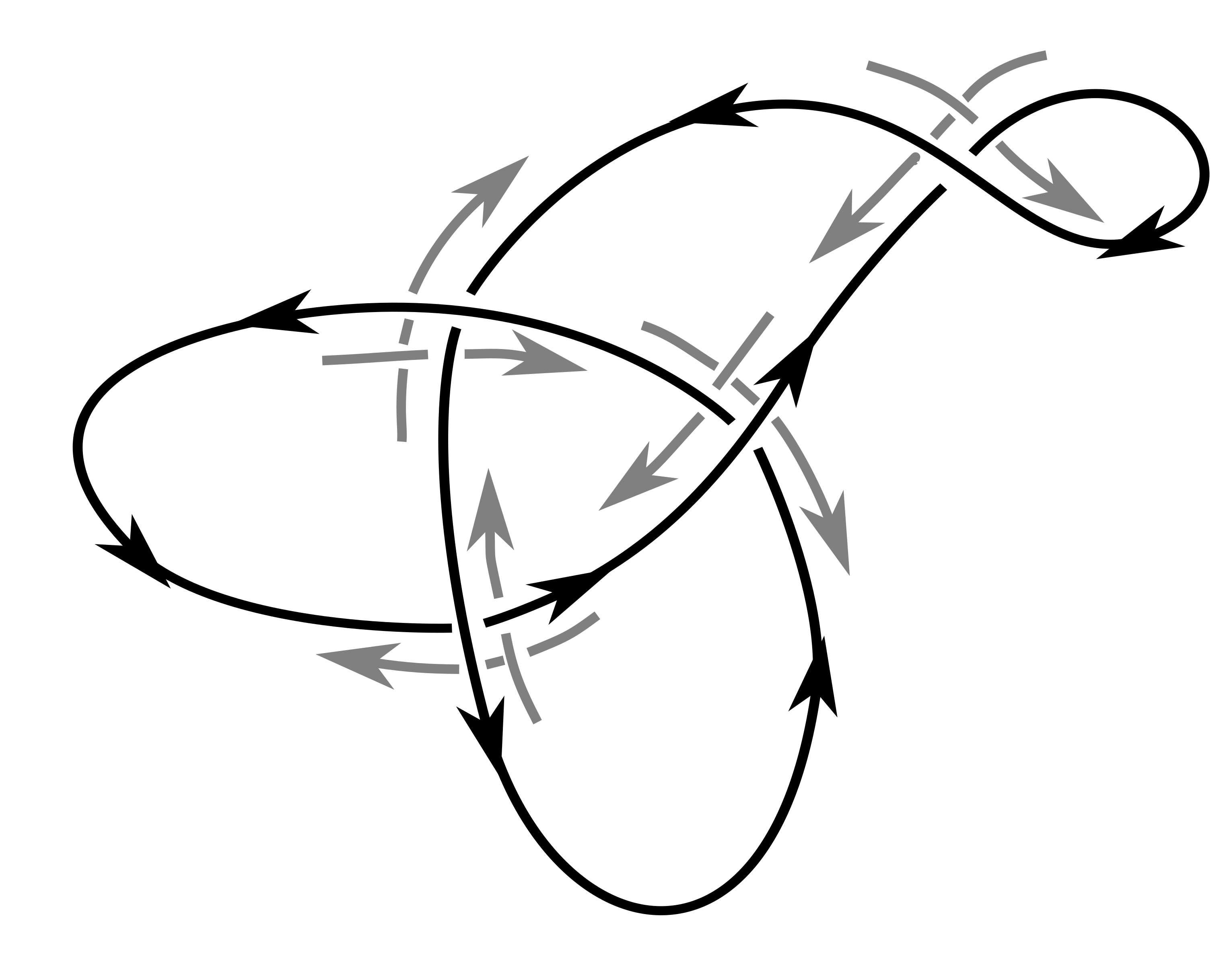}}
\put(230,40){$\longrightarrow$}
\put(260,0){\includegraphics[width=4cm]{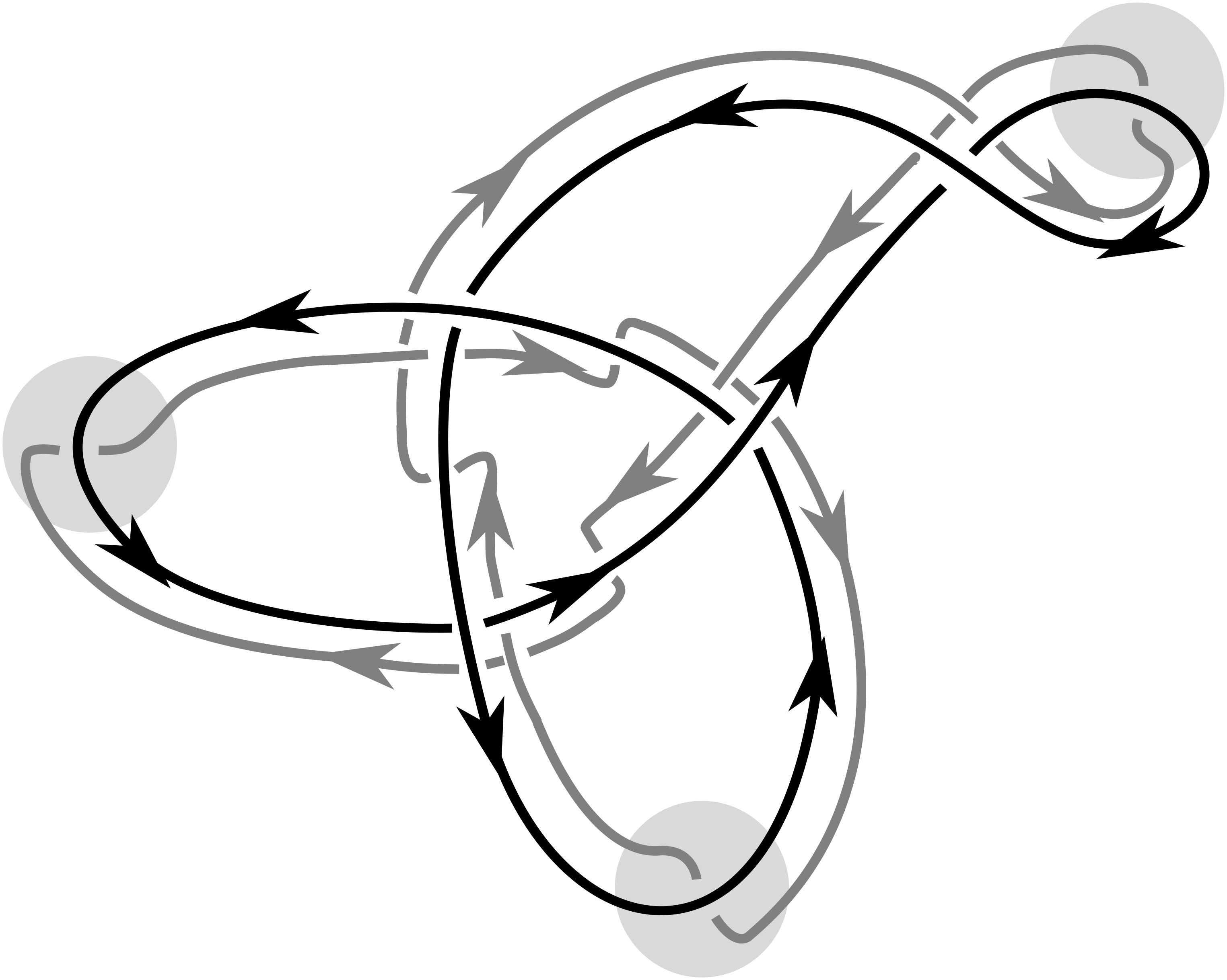}}
\end{picture}
\caption{A diagram for a trefoil $L$ and the corresponding diagram for $L\sqcup -L$. Three extra crossings occur along the edges of the original diagram, highlighted in light grey.}
\label{fig:LsqcupL}
\end{figure}
\end{enumerate}

Applying Seifert's algorithm to the resulting diagram of $L \sqcup -L$ yields a Seifert surface $\Sigma$ that deformation retracts onto a copy of $D$ where each vertex is replaced by a circle, as in figure ~\ref{fig:Seifert}, so the faces and vertices of $D$ correspond to a set of generators for $H_1(\Sigma)$. 

\begin{figure}[H]
\begin{picture}(350,130)
\put(0,0){\includegraphics[width=7cm]{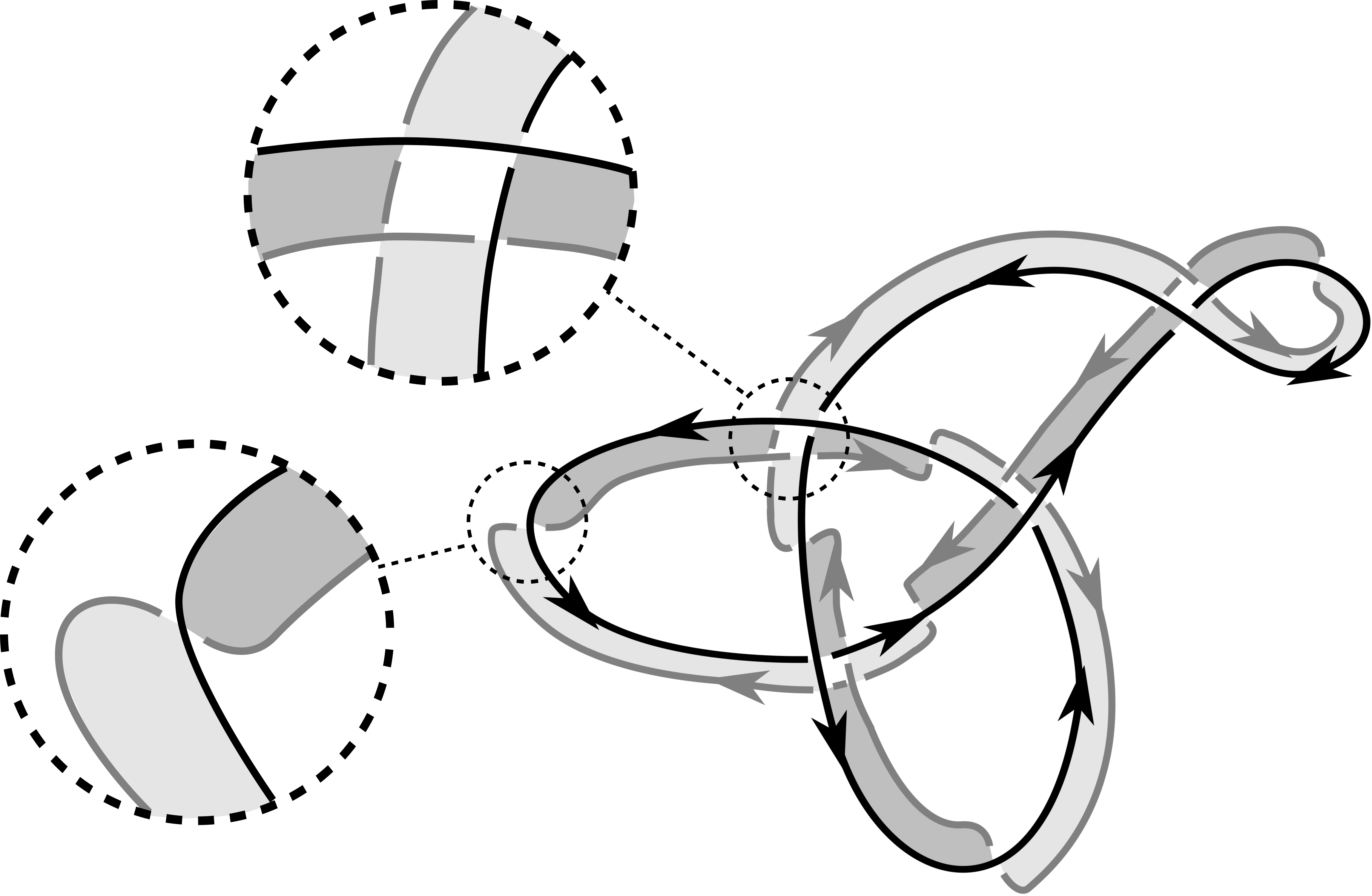}}
\put(210,40){$\longrightarrow$}
\put(240,0){\includegraphics[width=4cm]{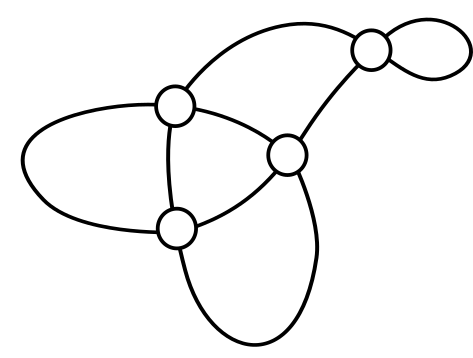}}
\end{picture}
\caption{A Seifert surface for $L \sqcup -L$ when $L$ is a trefoil, which deformation retracts onto a diagram of the trefoil, with vertices replaced by circles.}
\label{fig:Seifert}
\end{figure}

Let $\Sigma$ be a Seifert surface for $L \sqcup -L$ constructed as above from a diagram $D$, and let $A$ be a Seifert matrix with respect to the faces and vertices of $D$. Near a vertex of $D$ there are five homology classes: one for each surrounding face, and one for the vertex. We can consider all five at once using a single picture as in figure ~\ref{fig:hom}, where some segments in the picture are viewed as parts of different homology classes depending on the context. For example, in figure ~\ref{fig:hom} the top edge of the square in \ref{fig:homall} can be viewed as a part of the curve corresponding to the upper face (\ref{fig:homa}) or as a part of the curve corresponding to the vertex (\ref{fig:homx}). 

\begin{figure}[h!]
\begin{subfigure}[t]{0.3\textwidth}
\centering \includegraphics[trim = 0 110 0 110, clip,width = 2.5cm]{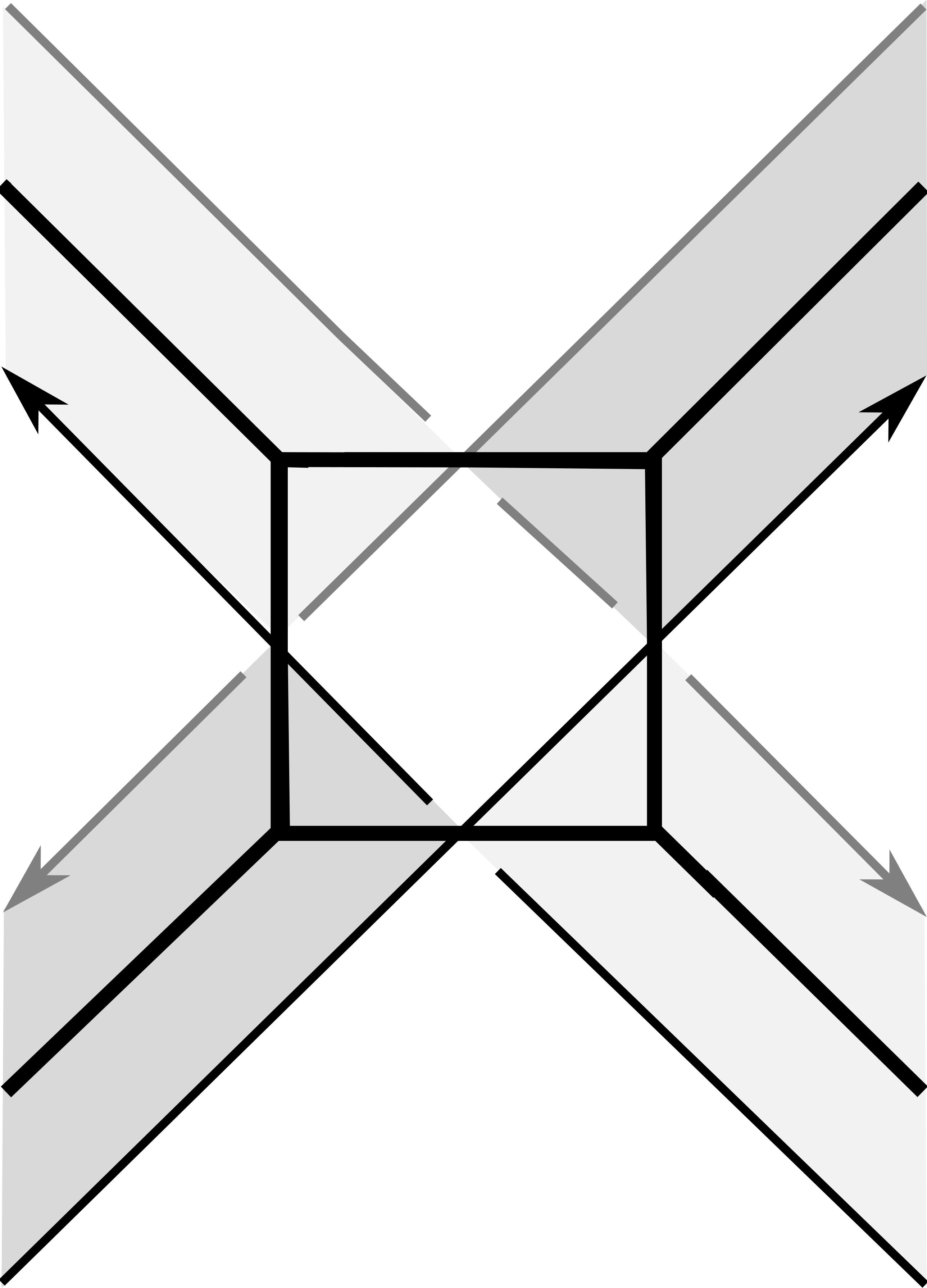}
\caption{All five homology curves near a vertex represented in one picture}
\label{fig:homall}
\end{subfigure}
\hspace*{0.5cm}
\begin{subfigure}[t]{0.3\textwidth}
\centering\includegraphics[trim = 0 110 0 110,clip, width = 2.5cm]{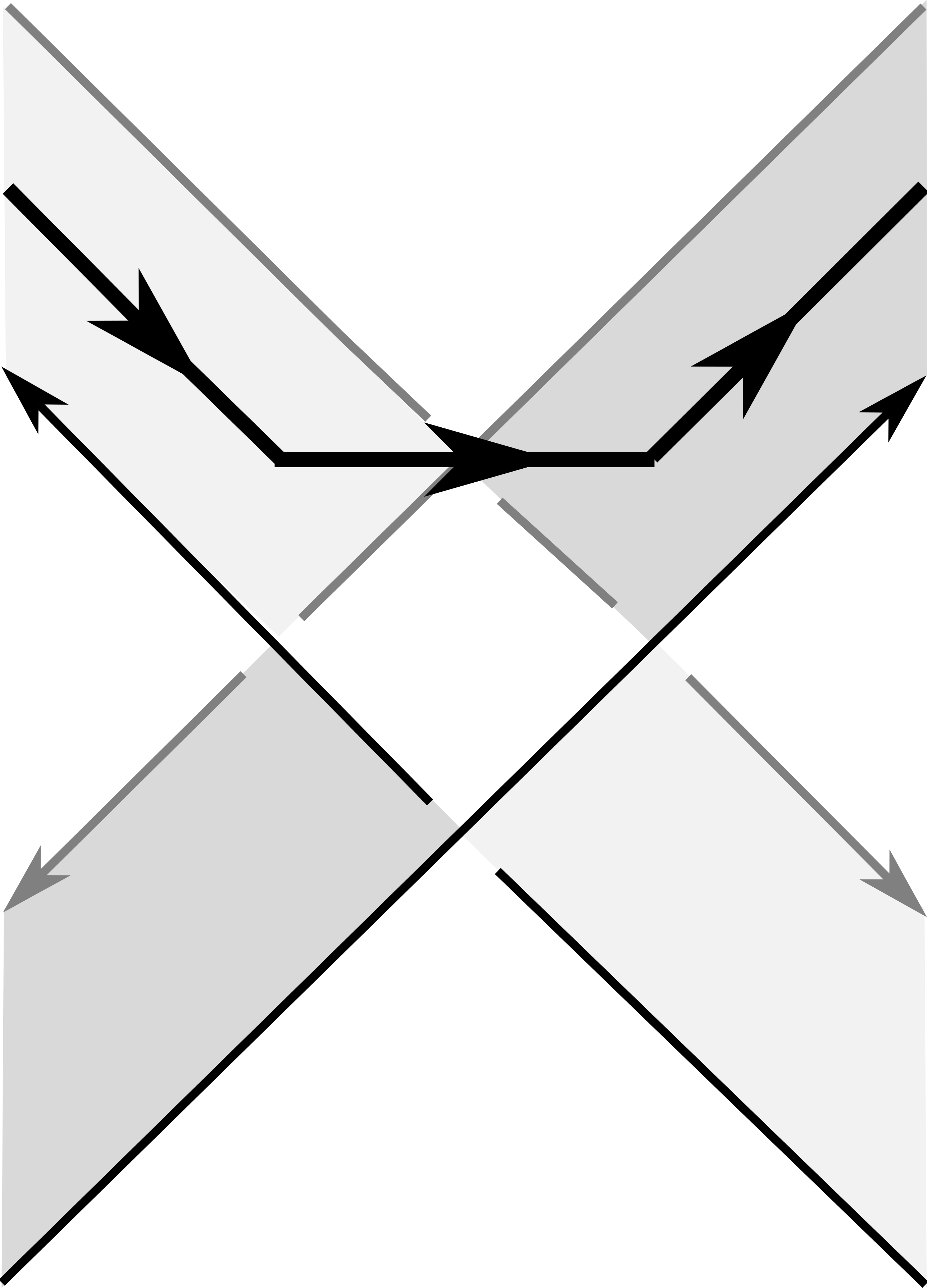}
\caption{Curve representing the homology class corresponding to the upper face}
\label{fig:homa}
\end{subfigure}
\hspace*{0.5cm}
\begin{subfigure}[t]{0.3\textwidth}
\centering\includegraphics[trim = 0 110 0 110,clip,width = 2.5cm]{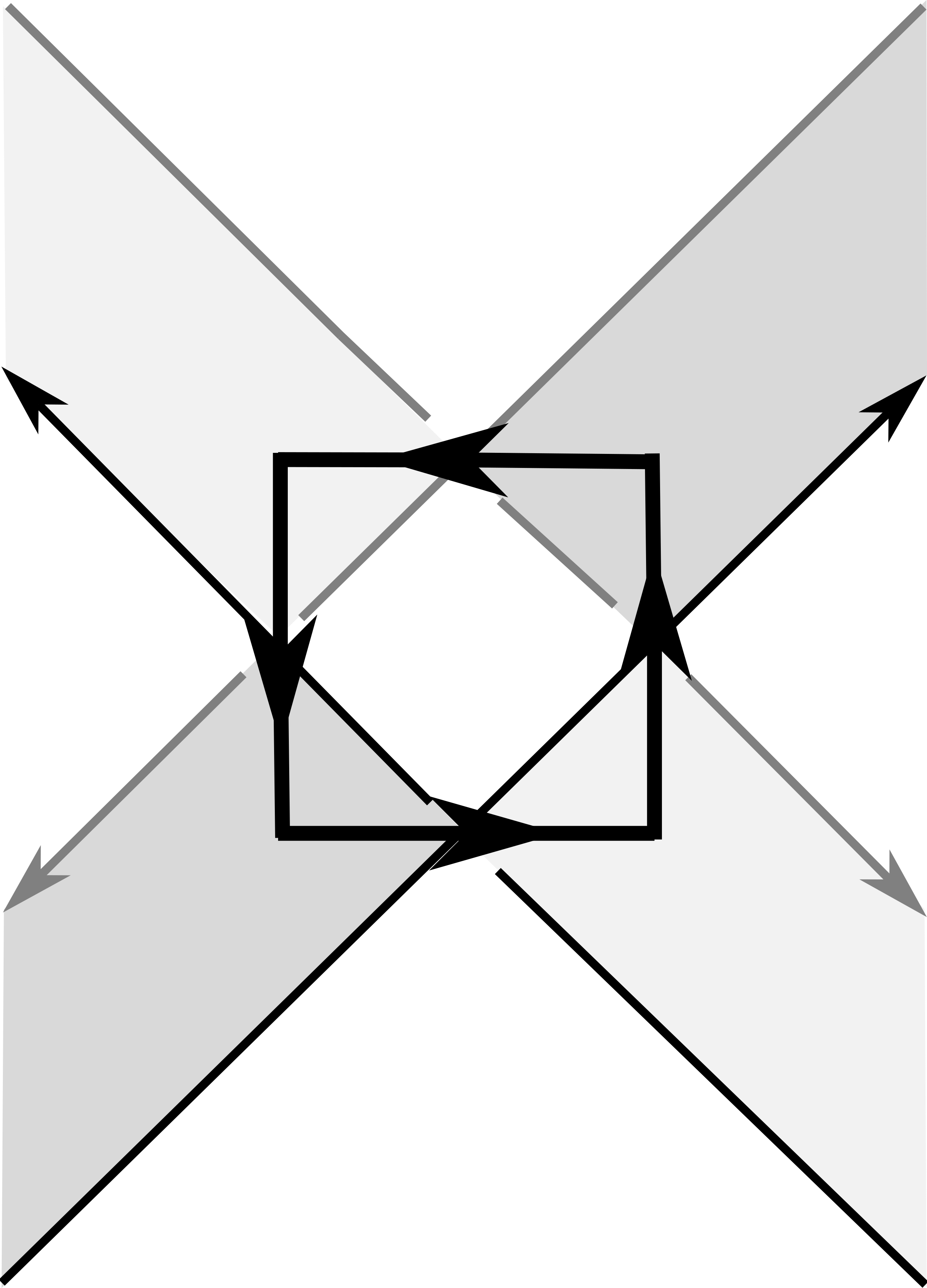}
\caption{Curve representing the homology class corresponding to the vertex}
\label{fig:homx}
\end{subfigure}

\caption{The homology classes (thick black) near a positive vertex of $D$}
\label{fig:hom}
\end{figure} Using the convention that all curves generating $H_1(\Sigma)$ are oriented counterclockwise in the diagram, we compute the local contribution to the Seifert matrix $A$ near each vertex of $D$, see figure ~\ref{fig:link}.

\begin{figure}[H]
\begin{subfigure}{1 \textwidth}
\centering
\begin{picture}(300,100)
\put(10,10){\includegraphics[trim = 0 110 0 110,clip,width = 3cm]{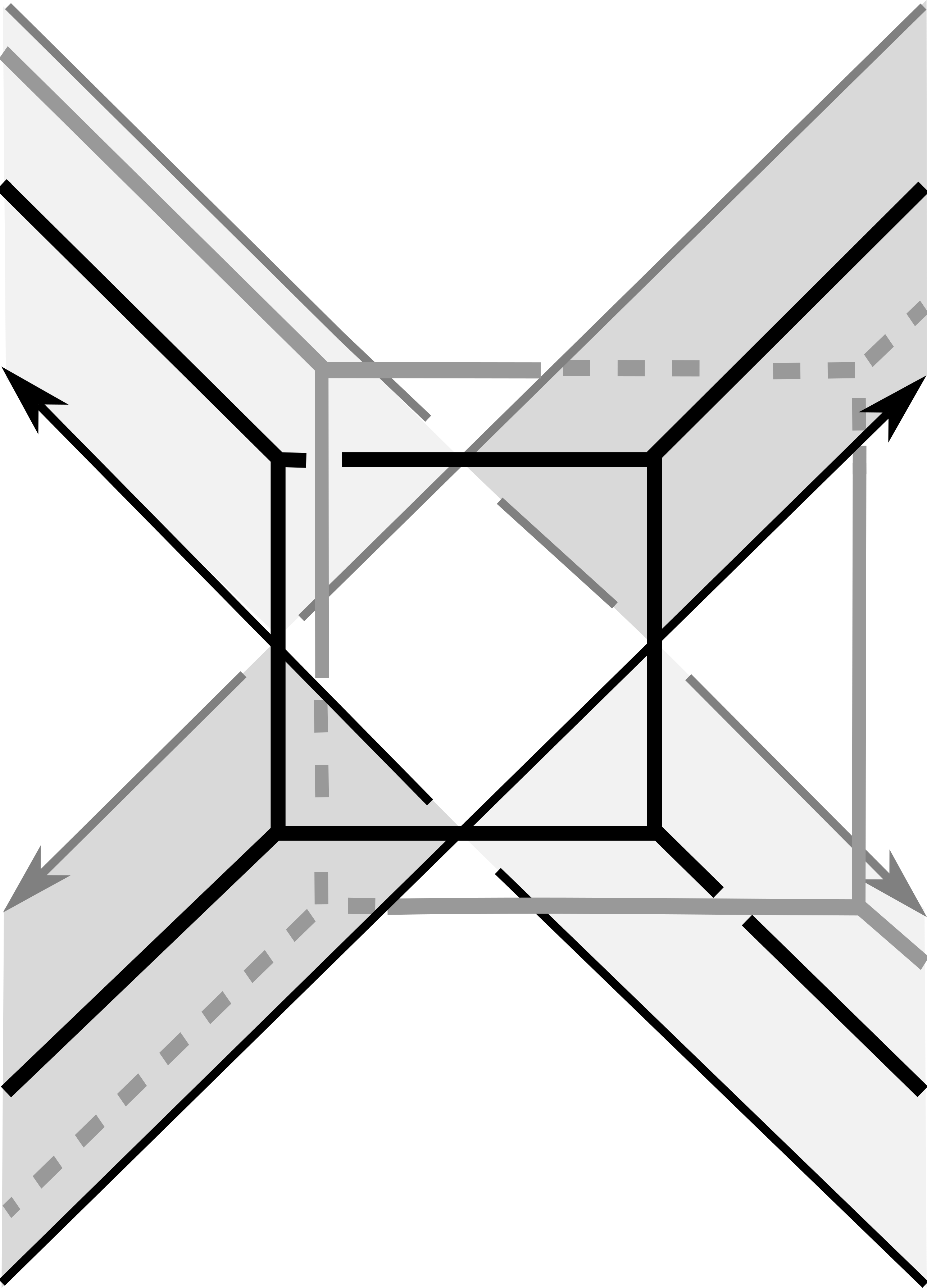}}
\put(47,97){$f_a$}
\put(44,87){\rotatebox[origin=c]{180}{\scalebox{1.5}[1.25]{$\curvearrowleft$}}}

\put(0,51){$f_b$}
\put(10,51){\rotatebox[origin=c]{270}{\scalebox{1.5}[1.25]{$\curvearrowleft$}}}

\put(47,7){$f_c$}
\put(44,15){\scalebox{1.5}[1.25]{$\curvearrowleft$}}

\put(100,51){$f_d$}
\put(93,50){\rotatebox[origin=c]{90}{\scalebox{1.5}[1.25]{$\curvearrowleft$}}}

\put(49,51){$v$}
\put(44,48){\rotatebox[origin=c]{0}{\scalebox{2}{$\circlearrowleft$}}}
\put(130,50){$\begin{array}{c||cccc|c}
\lk & f_a & f_b & f_c & f_d & \vtex\\
\hline
\hline
&&&&&\\[-0.3cm]
i^-(f_a) & -1/2 & 0 & 0 & 1/2 & 0\\
i^-(f_b) & -1/2 & 0 & -1/2 & 0 & 1\\
i^-(f_c) & 0 & 0 & -1/2 & 1/2 & 0\\
i^-(f_d) & 0 & 0 & 0 & 0 & 0\\
\hline
&&&&&\\[-0.3cm]
i^-(\vtex) &1 & 0 & 1 & -1 & -1
\end{array}$}
\end{picture}
\caption{Contribution to $A$ near a positive vertex}
\end{subfigure}
\end{figure}

\begin{figure}[H]
\ContinuedFloat
\begin{subfigure}{1 \textwidth}
\centering
\begin{picture}(300,100)
\put(10,10){\includegraphics[trim = 0 110 0 110, clip, width = 3cm]{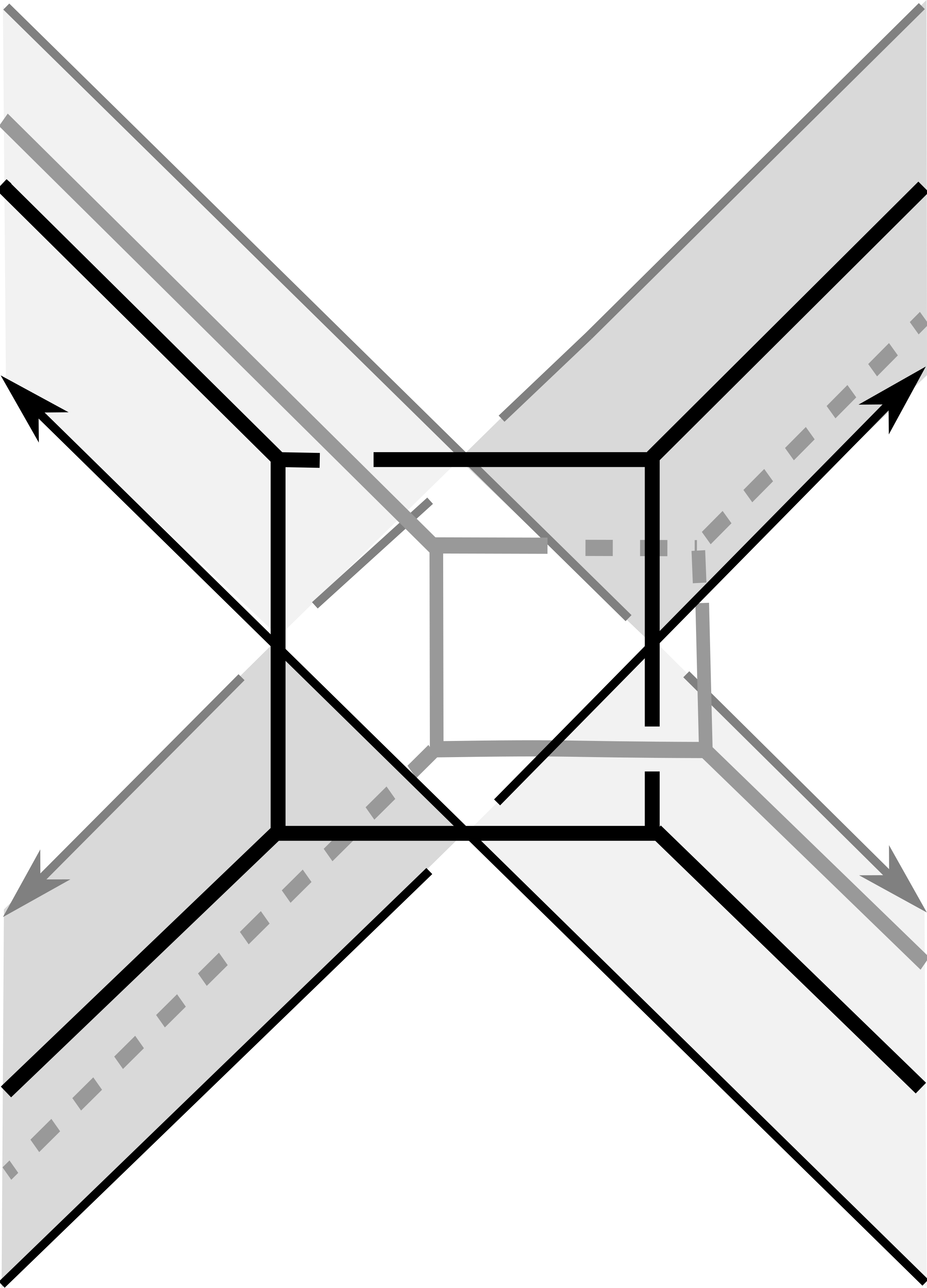}}
\put(47,97){$f_a$}
\put(44,87){\rotatebox[origin=c]{180}{\scalebox{1.5}[1.25]{$\curvearrowleft$}}}

\put(0,51){$f_b$}
\put(10,51){\rotatebox[origin=c]{270}{\scalebox{1.5}[1.25]{$\curvearrowleft$}}}

\put(47,7){$f_c$}
\put(44,15){\scalebox{1.5}[1.25]{$\curvearrowleft$}}

\put(100,51){$f_d$}
\put(93,50){\rotatebox[origin=c]{90}{\scalebox{1.5}[1.25]{$\curvearrowleft$}}}

\put(55,51){$v$}
\put(51,49){\rotatebox[origin=c]{0}{\scalebox{1.7}{$\circlearrowleft$}}}
\put(130, 50){$
\begin{array}{c||cccc|c}
\lk & f_a & f_b & f_c & f_d & \vtex\\
\hline
\hline
&&&&&\\[-0.3cm]
i^-(f_a) & 1/2 & 0 & 0 & 1/2 & -1\\
i^-(f_b) & -1/2 & 0 & -1/2 & 0 & 1\\
i^-(f_c) & 0 & 0 & 1/2 & 1/2 & -1\\
i^-(f_d) & 0 & 0 & 0 & 0 & 0\\
\hline
&&&&&\\[-0.3cm]
i^-(\vtex) &0 & 0 & 0 & -1 & 1
\end{array}$}
\end{picture}
\caption{Contribution to $A$ near a negative vertex}
\end{subfigure}
\caption{Contribution to a Seifert matrix $A$ near a vertex of $D$}
\label{fig:link}
\end{figure}

Away from the vertices of $D$, for example near the border of the diagrams in figure ~\ref{fig:link}, the computation takes the following convention: If the pushout is not visible (dotted thick grey), it is drawn between the homology curve (thick black) and the original diagram for $D$ (thin black), and if the pushout is visible (solid thick grey), it is drawn between the homology curve and the diagram for $-L$ (thin grey). This convention is chosen so that no linking comes from the crossings that occur along the edges of $D$, see figure ~\ref{fig:nolinking}.

\begin{figure}[H]
\centering
\includegraphics[trim = 0 30 0 0, clip, width=1.75cm]{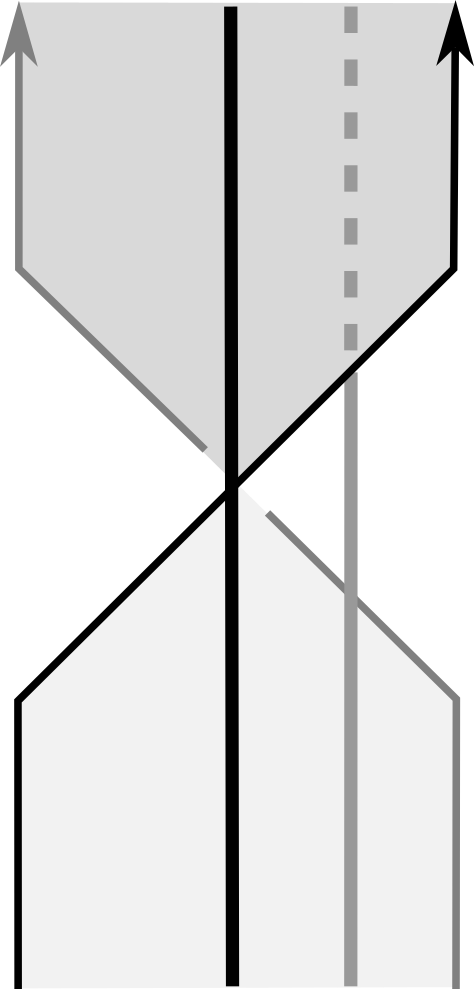}
\hspace*{2cm}
\includegraphics[trim = 0 30 0 0, clip, width=1.75cm]{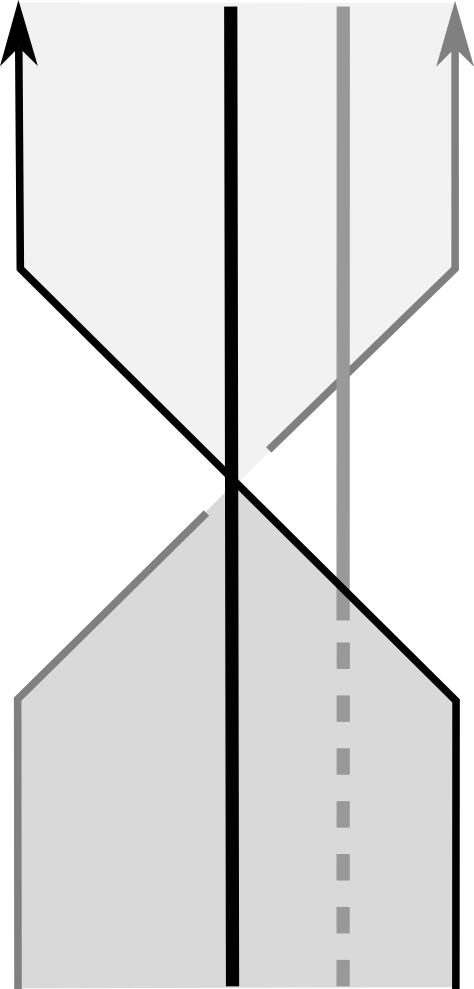}
\caption{No linking from crossings occuring along the edges $D$}
\label{fig:nolinking}
\end{figure}

The Seifert matrix $A$ can be computed by summing over vertices of $D$ using the local contributions given in figure ~\ref{fig:link}. Note that a curve around a vertex does not have any linking with the pushout of a curve around a different vertex, so the matrix $Q = (1-\om)A + (1-\bom)A^T$ has the following form: $$Q=\begin{array}{c||c|c}&\text{faces} & \text{vertices}\\
\hline
\hline
\text{faces} & X & Y \\
\hline
\text{vertices} & Y^* & Z
\end{array}$$ where $Z$ is a diagonal matrix with $-(1-\om) - (1-\bom)$ in the entries corresponding to positive vertices and $1-\om + 1-\bom$ in the entries corresponding to negative vertices. In particular, $Z$ is invertible whenever $\om \neq 1$ and $\sign(Z) = -\wri(D)$. Observe that $Q$ is congruent to the block diagonal matrix $$MQM^* = \begin{pmatrix}X-YZ^{-1}Y^* & 0 \\ 0 &  Z^{-1}\end{pmatrix}$$ via the matrix $$M =\begin{pmatrix} I & -YZ^{-1} \\ 0 & Z^{-1}\end{pmatrix}$$ As congruent matrices have the same signature and $\sign(Q) = 2\sigma_L$, it remains to show that $$\sign(X-YZ^{-1}Y^*) = \sign(\tau_D)$$ 

We show that $X-YZ^{-1}Y^*$ is exactly equal to $\tau_D$ under an appropriate scaling of the generators. Make the following computation $$(YZ^{-1}Y^*)_{i,j} = \sum_{k, \ell} Y_{i,k}(Z^{-1})_{k, \ell}(Y^*)_{\ell, j} = \sum_{k} Y_{i, k}Z_{k,k}^{-1}\overline{Y_{j, k}}$$ and notice that it's only possible for both $Y_{i,k}$ and $Y_{j, k}$ to be nonzero if faces $i$ and $j$ are both adjacent to vertex $k$. Therefore $YZ^{-1}Y^*$ is a sum over vertices, where the contribution at each vertex is a matrix that is zero everywhere except in the $4 \times 4$ minor corresponding to the four adjacent faces of the vertex. The same is then true for $X-YZ^{-1}Y^*$, and the $4 \times 4$ minor at each vertex is given by doing the corresponding matrix operations to the local contribution to $Q$. We show the explicit computation for this $4\times 4$ minor; the result is in figure ~\ref{fig:final}.

Writing the contribution to $A$ from figure ~\ref{fig:link} heuristically as $$\begin{array}{c||c|c}
\lk & f & \vtex\\
\hline
\hline
i^-(f) & \alpha & \beta\\
\hline
i^-(\vtex) & \gamma & \delta\end{array}$$ where $\alpha$ is a $4 \times 4$ block and $\delta$ is a $1\times 1$ block, and using the substitution $$s = 1-\om$$ we can write the local contribution to $Q=sA+\bar{s}A^T$ as follows: $$\begin{array}{c||c|c}
 & \bar{f} & \bar{\vtex}\\
\hline
\hline
f & s\alpha+s\alpha^T & s\beta+\bar{s}\gamma^T\\
\hline
c & s\gamma+\bar{s}\beta^T & (s+\bar{s})\delta\end{array}$$ The notation $\bar{\vtex}$ and $\bar{f}$ in the top row reminds us that when viewed as a Hermitian form, $Q$ is conjugate linear in the second component. The local contribution to $X-YZ^{-1}Y^*$ is therefore \begin{align*}
    (s\alpha+\bar{s}\alpha^T) - (s\beta +\bar{s}\gamma^T)((s+\bar{s})\delta)^{-1}(s\gamma+\bar{s}\beta^T)
\end{align*}
which, noting that $s+\bar{s} = s\bar{s}$ and $s\bom = -\bar{s}$, simplifies to $$s\alpha+\bar{s}\alpha^T - (\beta - \bom \gamma^T)\delta^{-1}(\beta^T -\om \gamma)$$ Using the values for $\alpha, \beta, \gamma,$ and $\delta$ from figure ~\ref{fig:link}, this is
\begin{figure}[H]\centering $$\left[\begin{array}{c||cccc}(+)&\bar{f_a} & \bar{f_b} & \bar{f_c} & \bar{f_d} \\
\hline
\hline
&&&&\\[-0.3cm]
f_a& \frac{\om+\bom}{2}& -\frac{1+\bom}{2} & 1 & -\frac{1+\om}{2} \\
&&&&\\[-0.3cm]
f_b&-\frac{1+\om}{2} & 1 & -\frac{1+\om}{2} & \om \\
&&&&\\[-0.3cm]
f_c& 1 &-\frac{1+\bom}{2} &\frac{\om+\bom}{2} & -\frac{1+\om}{2}\\
&&&&\\[-0.3cm]
f_d&-\frac{1+\bom}{2}  & -\bom & -\frac{1+\bom}{2}  & 1\end{array} \right] \hspace*{1cm} \left[\begin{array}{c||cccc}(-)&\bar{f_a} & \bar{f_b} & \bar{f_c} & \bar{f_d} \\
\hline
\hline
&&&&\\[-0.3cm]
f_a& -\frac{\om+\bom}{2} & \frac{1+\bom}{2} & -1 & 1+\om\\
&&&&\\[-0.3cm]
f_b & \frac{1+\om}{2} & -1 & \frac{1+\om}{2} & -\om \\
&&&&\\[-0.3cm]
f_c & -1 & \frac{1+\bom}{2} & -\frac{\om+\bom}{2} & \frac{1+\om}{2}\\
&&&&\\[-0.3cm]
f_d & \frac{1+\bom}{2} & \bom & \frac{1+\bom}{2} & -1 \end{array} \right]$$
\caption{The local contribution to $X-YZ^{-1}Y^*$ of a positive (left) and negative (right) vertex}
\label{fig:final}
\end{figure}

The local contributions to $X-YZ^{-1}Y^*$ in figure ~\ref{fig:final} become the local contributions to $\tau_D$ in figure ~\ref{fig:kash} after scaling the faces in the following way. Let the winding of a face in $D$ be the winding number of $D$ around any point in that face. Extend the coefficient ring to $\R$, and for a face $f$ with winding $k$, use the generator $(-\bsom)^kf$ in place of $f$. This multiplies each row of $X-YZ^{-1}Y^*$ by some $(-\bsom)^k$ and the corresponding column  by the complex conjugate $(-\som)^{k} = (-\bsom)^{-k}$. The matrices in figure ~\ref{fig:final} become exactly the ones in figure ~\ref{fig:kash} under the identification $2x = \som+\bsom$, so $X-YZ^{-1}Y^*$ becomes $\tau_D$. Since this scaling preserves the signature, $\sign(X-YZ^{-1}Y^*) = \sign(\tau_D)$. \\
\end{proof}

\begin{remark} Note that theorem ~\ref{thm:kash} does not give an interpretation for Kashaev's invariant when $|x|>1$. In particular, we do not know if the Kashaev invariant for $|x|>1$ also correspond to values of $\sigma_L$, or if so which ones.\\
\end{remark}

\section{The Alexander polynomial and other consequences}

In this section we use Kashaev's matrix $\tau_D$ to give another formula for $\Delta_L$ and provide explicit formulas for the kernel of $\tau_D$. The formula for the Alexander polynomial appears in ~\cite{liv}, was independently noticed by Bar-Natan, and was the inspiration behind the main result of this paper. We give an alternate proof for this formula in corollary ~\ref{cor:alex}.

\begin{cor} \label{cor:alex} Let $f_i$ and $f_j$ be two faces in $D$ that share an edge, and let $\krem$ be the Kashaev matrix $\tau_D$ with the two columns and rows corresponding to $f_i$ and $f_j$ removed. The Alexander polynomial $\Delta_L$ is given by $$\Delta_L(t)^2 = \det(\krem)$$ where $2x = t^{1/2} + t^{-1/2}$. Furthermore, while the Alexander polynomial is only defined only up to multiplication by $\pm t$, the equality above is (up to a sign) a proper equality if $\Delta_L$ is taken to be the Conway-normalized Alexander polynomial.
\end{cor}

\begin{proof}
Consider any connect sum $L \# -L$ in place of $L \sqcup -L$. Use the same procedure as in the proof of theorem ~\ref{thm:kash} to construct a diagram and Seifert surface $\Sigma_{\#}$ for $L \# -L$, adding in the connect sum along the shared edge of $f_i$ and $f_k$.  

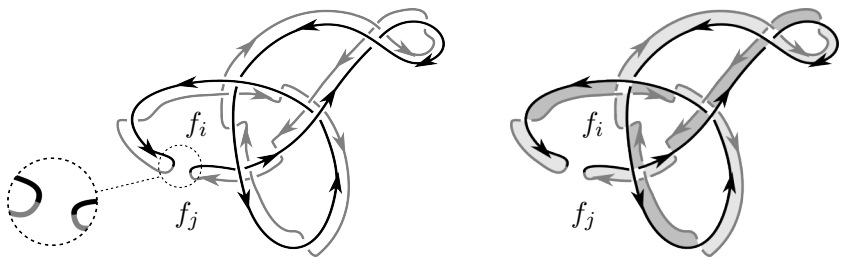
\begin{figure}[H]
\centering
\begin{picture}(250, 95)
\put(-40, 0){\includegraphics[height = 3.3cm]{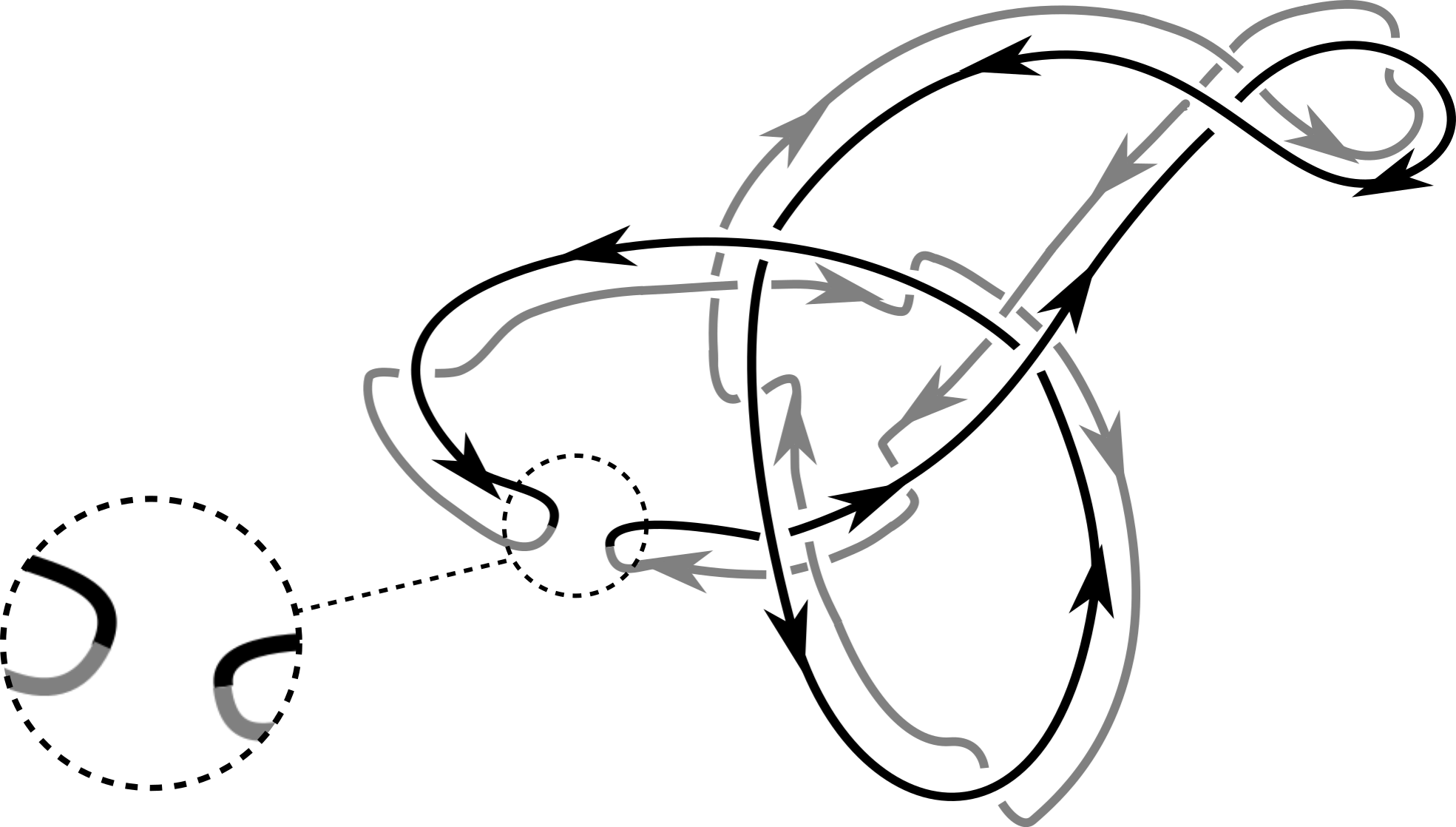}}
\put(27,47){$f_i$}
\put(23,13){$f_j$}
\put(150, 0){\includegraphics[height=3.3cm]{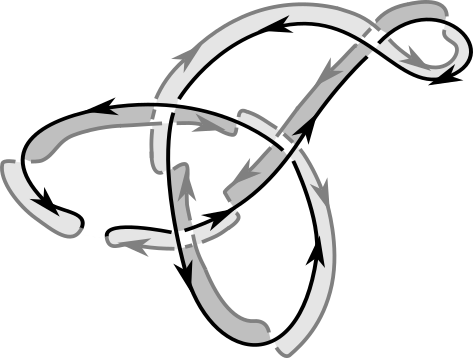}}
\put(177,47){$f_i$}
\put(173,13){$f_j$}
\end{picture}
\caption{A diagram and Seifert surface for $L \#-L$ with the connect sum along the shared edge of $f_i$ and $f_j$}
\end{figure}
Note that the faces and vertices of $D$ with $f_i$ and $f_j$ removed form a linearly independent basis for $H_1(\Sigma_{\#})$. Let $A$ be a Seifert matrix for $\Sigma_{\#}$ with respect to this basis. Note that $A$ is simply the Seifert matrix for $L \sqcup -L$ from the proof of ~\ref{thm:kash}, but with the two rows and columns corresponding to $f_i$ and $f_j$ removed.
Now consider the matrix $$Q = (1-t)A +(1-t^{-1})A^T$$ over the involutive ring $\mb{Z}[t^{\pm 1}]$ with involution given by $t\mapsto t^{-1}$. As in the proof of theorem ~\ref{thm:kash}, $Q$ is a block matrix of the following form $$Q=\begin{array}{c||c|c}&\text{faces except }  \{f_i, f_j\} & \text{vertices}\\\hline
\hline
\text{faces except }   \{f_i, f_j\} & X & Y \\
\hline
\text{vertices} & Y^* & Z
\end{array}$$ 
where the notation $Y^*$ denotes the transpose of $Y$ under the involution $t \mapsto t^{-1}$ so that when $t$ is a unit complex number, $Y^*$ is the usual conjugate transpose. Notice that $$MQM^* = \begin{pmatrix}X-YZ^{-1}Y^* & 0 \\ 0 &  Z^{-1}\end{pmatrix}$$ where $M =\begin{pmatrix} I & -YZ^{-1} \\ 0 & Z^{-1}\end{pmatrix}$ and recall that $Z$ is diagonal with $-((1-t)+(1-t^{-1}))$ in the entries corresponding to positive vertices and $(1-t)+(1-t^{-1})$ in the entries corresponding to negative vertices. Therefore \begin{align*}
    \det(Q) &= \det(X-YZ^{-1}Y^*)\det(Z^{-1})\det(M)^{-1}\det(M^*)^{-1}\\
    &=\det(X-YZ^{-1}Y^*)\det(Z)\\
    &=(-1)^p((1-t)+(1-t^{-1}))^{|V_D|}\det(X-YZ^{-1}Y^*)
\end{align*} where $p$ is the number of positive vertices in $D$ and $|V_D|$ is the total number of vertices in $D$. To change $X-YZ^{-1}Y^*$ into $\krem$ we multiply each row by some $(-t^{1/2})^{k}$ and the corresponding the column by $(-t^{-1/2})^{-k}$ so the determinant is preserved. We can also rewrite $(1-t) + (1-t^{-1})$ as $-(t^{1/2} + t^{-1/2})^2$, so we get $$\det(Q) = (-1)^{p+|V_D|}(t^{1/2} + t^{-1/2})^{2|V_D|} \det(\tau_D')$$ Finally, notice that $$Q = (t^{-1/2} + t^{1/2})(t^{1/2} A - t^{-1/2} A^T)$$ and use the Conway-Alexander formula for $\Delta_{L \# - L}$ to get \begin{align*}
\Delta_{L \# - L}(t) &= \det(t^{1/2} A -t^{-1/2} A^T)\\
&= (t^{-1/2}-t^{1/2})^{-2|V_D|} \det(Q)\\
&=(-1)^{p+|V_D|}\det(\krem)
\end{align*} 
Since $\Delta_{L \# -L} = \Delta_L^2$, this concludes the proof.\\
\end{proof}

\begin{remark}
We used an explicit formula for the Conway-Alexander polynomial in corollary ~\ref{cor:alex} to show that it is equal to $\det(\krem)$, but showing that $\det(\krem)$ gives \emph{any} Alexander polynomial would be enough to show that it gives the Conway-normalized one -- since $\krem$ is a real matrix when $t \in S^1$, it must be symmetric under $t \mapsto t^{-1}$.
\end{remark}

\begin{remark}The proof of corollary ~\ref{cor:alex} allows us to describe $\krem$ as a presentation matrix of the Alexander module of $L$ over the localized ring $\mb{Z}[t^{\pm 1/2}, (1-t)^{-1}]$ under the substitution $2x = t^{1/2}+t^{-1/2}$. If $L$ is a knot, multiplication by $1-t$ is invertible in the Alexander module, so $\krem$ also presents the Alexander module over $\mb{Z}[t^{\pm 1/2}]$. Note that this does \emph{not} describe the $\mb{Z}[2x]$-module that $\tau_D$ presents, or even show that such a module is an invariant of links. 
\end{remark}

We conclude with explicit formulas for the kernel of the Kashaev matrix $\tau_D$. We can view $\tau_D$ either as a linear map or as a symmetric bilinear form $\langle \cdot, \cdot \rangle_D$ on the $\mb{Z}[2x]$-module freely generated by the faces $F_D$ of $D$, where $\langle f_i, f_j \rangle = (\tau_D)_{i, j}$. The kernel of $\tau_D$ as a linear map is the same as its kernel as a symmetric bilinear form. The next corollaries give explicit formulas for this kernel.

\begin{cor} \label{cor:ker} The kernel of $\tau_D$ contains the 2-dimensional submodule generated by $$\sum_{f \in F_D} a_{w(f)}f$$ where $w(f)$ is the number of times the diagram winds around a point in the face $f$, and the coefficients $a_n$ are solutions to the recurrence relation $a_n + 2xa_{n+1}+a_{n+2} = 0$. If we use the identification  $2x = t^{-1/2} + t^{-1/2}$ and consider $\tau_D$ over the field $\mb{Q}(t^{1/2})$, then this subspace is the entire kernel if $\Delta_L(t) \neq 0$. Furthermore, the solutions to this recurrence are given explicitly by $$a_n = c_1(-t^{1/2})^n + c_2(-t^{1/2})^n$$ for constants $c_1, c_2 \in \mb{Q}(t^{1/2})$.
\end{cor}

\begin{proof}We first verify that $g = \sum_{f \in F_D} a_{w(f)}f$ lives in the kernel of the symmetric bilinear form $\langle \cdot, \cdot\rangle_D$ represented by $\tau_D$. Consider $\langle g, f\rangle_D$ for an arbitrary face $f$, and recall the local contributions in the definition of $\tau_D$ from figure ~\ref{fig:kash}, copied below.

\begin{figure}[H]
\begin{picture}(100,60)
\put(-50,0){\includegraphics[width = 2cm]{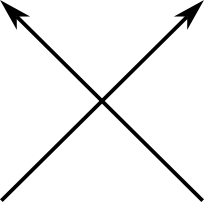}}
\put(-27,5){$f_c$}
\put(-27, 50){$f_a$}
\put(-50, 27){$f_b$}
\put(-5, 27){$f_d$}
\put(50, 30){$\begin{array}{c|cccc} & f_a & f_b & f_c & f_d\\
\hline
f_a&2x^2-1 & x & 1 & x\\
f_b&x&1&x&1\\
f_c&1&x&2x^2-1&x\\
f_d&x&1&x&1\end{array}$}
\end{picture}
\end{figure} 

If $f$ appears as $f_b$ in the diagram above near some vertex and $w(f)=k$, then $w(f_a)=w(f_c) = k-1$ and $w(f_d) = k-2$, so the contribution to $\langle g, f\rangle_D$ of this vertex is \begin{align*}&\, a_k\langle f_b, f_b\rangle_D+a_{k-1}\langle f_b, f_a\rangle_D+a_{k-1}\langle f_b, f_c\rangle_D+a_{k-2}\langle f_b, f_d\rangle_D \\
= & \, a_k + 2xa_{k-1}+a_{k-2}\\
= & \,0\end{align*}
A similar computation shows that the contribution is also zero when $f$ appears in position $f_a, f_c,$ or $f_d$, hence $g$ is in the kernel of $\tau_D$. The submodule is 2-dimensional since the solution space to $a_n + 2xa_{n+1}+a_{n+2} = 0$ is 2-dimensional. The proof of ~\ref{cor:alex} shows that the kernel of $\tau_D$ is also 2-dimensional when $\Delta_L \neq 0$, so over a field, this is the entire kernel. It is straightforward to verify that the explicit form of $a_n$ solves the recurrence.\\
\end{proof}

If $D$ is disconnected then the Alexander polynomial is always zero and the kernel is larger. We can extend corollary ~\ref{cor:ker} to give a more general result.

\begin{cor}Let $D$ be a diagram with $n$ connected components $D_1, \cdots, D_n$. Let $f_0$ be the exterior face, and let $F_{D_i}'$ be the set of \emph{interior} faces of $D_i$. The kernel of $\tau_D$ contains the $n+1$-dimensional subspace generated by $$a_0f_0+\sum_{i=1}^n\sum_{f \in F_{D_i}'} a_{i, w(f)}f$$ where each sequence $\{a_{i, k}\}_{k \in \mb{Z}}$ satisfies the recurrence relation $a_{i, k} + 2xa_{i, k+1}+a_{i, k+2}=0$ with the condition $a_{i,0} = a_0$ for all $i$. If none of the Alexander polynomials $\Delta_{L_i}(t)$ are zero, where $L_i$ is the link represented by $D_i$, then over the field $\mb{Q}(t^{1/2})$ this is the entire kernel.
\end{cor}
\begin{proof} The computations in the proof of corollary ~\ref{cor:ker}, along with the observation that interior faces of different diagrams don't share vertices, show that the subspace is indeed in the kernel. The solution space to the recurrence is $n+1$ dimensional, so it remains to show that the kernel of $\tau_D$ has dimension $n+1$ when all the $\Delta_{L_i}$ are nonzero. Using a similar procedure as in the proof of corollary ~\ref{cor:alex}, pick one face in each $F_{D_i}'$ that shares an edge with $f_0$, and let $\krem$ be $\tau_D$ with the $n+1$ rows and columns corresponding to these faces and $f_0$ removed. Let $A_i$ be a Seifert matrix for $L_i \# -L_i$ so that the block diagonal matrix with $tA_i-A_i^T$ in each block has determinant $\prod_{i=1}^n \Delta_{L_i}$. As in the proof of ~\ref{cor:alex}, we can get to $\krem$ from this block diagonal matrix without changing the determinant, so if each $\Delta_{L_i}$ is nonzero, $\krem$ has full rank and the kernel of $\tau_D$ has dimension $n+1$.\\
\end{proof}

\bibliographystyle{alpha}
\bibliography{main.bib}{}

\end{document}